\documentclass[reqno]{amsart}
\usepackage[english]{babel}
\usepackage{amscd,amssymb,amsmath,amsfonts,latexsym,amsthm}
\usepackage{inputenc}
\usepackage{graphicx}

\newtheorem{thm}{Theorem}[section]
\newtheorem{cor}[thm]{Corollary}

\newtheorem{prop}[thm]{Proposition}
\newtheorem{defn}[thm]{Definition}
\newtheorem{exam}[thm]{Example}
\newtheorem{rem}[thm]{Remark}

\numberwithin{equation}{section}
\DeclareFontFamily{OT1}{pzc}{}
\DeclareFontShape{OT1}{pzc}{m}{it}{<-> s * [1.10] pzcmi7t}{}
\DeclareMathAlphabet{\mathpzc}{OT1}{pzc}{m}{it}

\DeclareMathOperator{\Der}{Der}
\DeclareMathOperator{\Hom}{Hom}
\DeclareMathOperator{\Orb}{Orb} 
 
\DeclareMathOperator{\GL}{GL} \DeclareMathOperator{\Ann}{Ann}
 \DeclareMathOperator{\tr}{tr}
\DeclareMathOperator{\ad}{ad} \DeclareMathOperator{\gr}{gr}
\DeclareMathOperator{\Rad}{Rad} \DeclareMathOperator{\spann}{span}
\DeclareMathOperator{\Lie}{Lie}
\DeclareMathOperator{\Leib}{\mathpzc{Leib}}

\begin{document}


\title[Some irreducible components of the variety of Leibniz algebras]
{Some irreducible components of the variety of complex $n+1$-dimensional Leibniz algebras}
\author{A.Kh. Khudoyberdiyev}
\address{[A.Kh. Khudoyberdiyev] Institute of Mathematics, National University of Uzbekistan, Tashkent, 100125, Uzbekistan.} \email{khabror@mail.ru}
\author{M. Ladra}
\address{[M. Ladra] Department of Algebra, University of Santiago de Compostela, 15782, Spain.}
\email{manuel.ladra@usc.es}
\author{K.K. Masutova}
\address{[K.K. Masutova] Institute of Mathematics, National University of Uzbekistan, Tashkent, 100125, Uzbekistan.} \email{kamilyam81@mail.ru}
\author{B.A. Omirov}
\address{[B.A. Omirov] Institute of Mathematics, National University of Uzbekistan, Tashkent, 100125, Uzbekistan.}
\email{omirovb@mail.ru}


\begin{abstract}

In the present paper we indicate some Leibniz algebras whose
closures of orbits under the natural action of $\GL_n$ form an irreducible
component of the variety of complex $n$-dimensional Leibniz
algebras. Moreover, for these algebras we calculate the bases of
their  second groups of cohomologies.
\end{abstract}



\maketitle

\section{Introduction}

Jean-Louis Loday introduced Leibniz algebras because of
considerations in algebraic K-theory \cite{Loday}. We know that the
Lie algebra homology involves the Chevalley-Eilenberg chain
complex, which in turns involves  exterior powers of the Lie
algebra. Loday found that there is a non-antisymmetric generalization
where roughly speaking one has the tensor and not the exterior
powers of the Lie algebra in the complex; this new complex defines
the Leibniz homology of Lie algebras. The Leibniz homology is
related to the Hochschild homology in the same way the Lie algebra
homology is related to the cyclic homology.

In many cases where the Leibniz algebra involved may depend on the parameters it is useful to know the structure of the set of all Leibniz algebras of a given dimension.
The aim of this work  is to establish some results  from a geometrical point of view in the
study of Leibniz algebras. Any Leibniz algebra law is considered
as a point of an affine algebraic variety defined by the polynomial equations coming from
the Leibniz identity for a given basis. This way provides an description of the difficulties
in classification problems referring to the classes of nilpotent and solvable Leibniz algebras.
 The orbits relative to the action of the
general linear group correspond to  the isomorphism classes  of Leibniz
algebras  and so classification problems (up to isomorphism) can be reduced to the
classification of these orbits.
An affine algebraic variety is a union of a finite number of
irreducible components and the Zariski open orbits provide interesting classes of Leibniz algebras
to be classified. The Leibniz algebras of this class are called rigid.

The research of varieties of Lie algebras
laws over the field $\mathbb{C}$ complex numbers have been extensively studied, establishing various important structural results and properties.
On the contrary, the problem for varieties of Leibniz algebras  has not been considered in detail. The research of varieties of Lie and Leibniz algebra
laws is essentially based on the cohomological study of Leibniz algebras and on deformation
theory. Deformations of arbitrary rings and associative algebras, results about rigid Lie algebras and related cohomology questions,
 were first investigated in 1964 by Gerstenhaber \cite{Gersten}.
Later, the notion of deformation was applied to Lie algebras by Nijenhuis and Richardson \cite{Nijen},
 where they transform the topological problem related to rigidity into a cohomological problem,
  proving that a Lie algebra  $\mathfrak{g}$ is rigid if the second group $H_2(\mathfrak{g}, \mathfrak{g})$ of the Chevalley-Eilenberg cohomology vanishes.

 In this paper, we are concerned with the structure of the variety $\Leib_{n+1}$, the variety of the $(n+1)$-dimensional Leibniz algebras,
  in particular, with answers to the following question: What irreducible components do $\Leib_{n+1}$ fall into? The answers to this question
   would allow to describe partially the structures of some  Leibniz algebras of dimension $n+1$.
  We shall obtain  general results on some irreducible components of the variety of finite-dimensional Leibniz algebras
  and indicate representatives of solvable Leibniz algebras, whose closures of orbits form irreducible components.
 We hope to develop this line of research in the next works.

The paper is organized as follows. In Section~\ref{S:prel} we  recall some necessary notions about  Leibniz algebras, cohomology and degenerations of  Leibniz algebras.
In Section~\ref{S:irred} we describe derivations of the solvable Leibniz algebras whose nilradical
is a filiform algebra of type  $F_n^1$ (see below Theorem~\ref{T:fil}), present (1,1)-invariants for
various types of solvable algebras and give representatives of irreducible components  of the variety $\Leib_{n+1}$ of the $(n+1)$-dimensional Leibniz algebras.
Finally, in the last subsection, give  descriptions of the second cohomology group of the solvable Leibniz algebras whose nilradical is a filiform algebra of type  $F_n^1$.

Throughout of the paper,  we denote by $L$ a finite-dimensional Leibniz algebra
over the field of complex numbers. Moreover, in the  multiplication
table of a Leibniz algebra the omitted products and in the
expansion of 2-cocycles the omitted values are assumed to be zero.

\section{Preliminaries}\label{S:prel}

In this section we give necessary definitions on Leibniz algebras, cohomology,
degenerations  and known results.

We present the definition of the main object of our study.
\begin{defn} \cite{Loday}
A Leibniz algebra over a field $\mathbb{F}$ is a vector space
$L$ equipped with a bilinear map, called bracket,
\[[-,-] \colon  L \times  L \rightarrow  L \]
satisfying the Leibniz identity:
\[ [x,[y,z]]=[[x,y],z]-[[x,z],y], \]
for all $x,y,z \in  L$.
\end{defn}

The set $\Ann_r(L)=\{x \in L \ | \ [y,x]=0, \ \forall y \in L\}$ is
called \emph{the right annihilator of the Leibniz algebra $L$}.
Note that $\Ann_r(L)$ is an ideal of $L$ and for any $x, y \in L$,
the elements $[x,x]$, $[x,y]+ [y,x]\in \Ann_r(L)$.

\subsection{Solvable Leibniz algebras}

For a Leibniz algebra $L$ we consider the following \emph{central lower}
and \emph{derived series}:
\begin{align*}
L^1 & =L,\quad L^{k+1}=[L^k,L^1],    \ \quad \qquad k \geq 1; \\
L^{[1]}  & = 1,  \ \quad L^{[s+1]} = [L^{[s]}, L^{[s]}], \qquad s \geq 1.
\end{align*}
\begin{defn} \label{defn22}
A Leibniz algebra $L$ is said to be nilpotent (respectively,
solvable), if there exists $n\in\mathbb N$ ($m\in\mathbb N$) such
that $L^{n}=0$ (respectively, $L^{[m]}=0$). The minimal number $n$
(respectively, $m$) with such property is said to be the index of
nilpotency (respectively, of solvability) of the algebra $L$.
\end{defn}

Obviously, the index of nilpotency of an $n$-dimensional nilpotent
Leibniz algebra is not greater than $n+1$.

\begin{defn} An $n$-dimensional Leibniz algebra is said to be null-filiform if $\dim L^i=n+1-i, \ 1\leq i \leq n+1$.
\end{defn}

Remark that a null-filiform Leibniz algebra has maximal index of
nilpotency.

\begin{thm}[\cite{Ayup}] \label{thm24} An arbitrary $n$-dimensional null-filiform Leibniz algebra is isomorphic to the algebra:
\[NF_n: \quad [e_i, e_1]=e_{i+1}, \qquad 1 \leq i \leq n-1,\]
where $\{e_1, e_2, \dots, e_n\}$ is a basis of the algebra $NF_n$.
\end{thm}

From Theorem~\ref{thm24} it is easy to see that a nilpotent
Leibniz algebra is null-filiform if and only if it is a
one-generated algebra, i.e., an algebra generated by unique
element. Note that this notion has no sense in the Lie algebras
case, because they are at least two-generated.

It should be noted that  the sum of any two nilpotent (solvable)
ideals is nilpotent (solvable).

\begin{defn} The  maximal nilpotent (solvable) ideal of a Leibniz algebra is said to be a nilradical (solvable radical) of the algebra.
\end{defn}

Below, we present the  description of solvable Leibniz algebras
whose nilradical is isomorphic to the algebra $NF_n$.

\begin{thm}[\cite{Cas1}] \label{thm26} Let $R$ be a solvable Leibniz algebra whose nilradical is $NF_n$.
Then there exists a basis $\{e_1, e_2, \dots, e_n, x\}$ of the
algebra $R$ such that the multiplication table of $R$ with respect
to this basis has the following form:
\[RNF_n:\left\{ \begin{aligned}
{}[e_i,e_1] & =e_{i+1}, && 1\leq i\leq n-1,\\
 [x,e_1]& =-e_1, && \\
[e_i,x] & =ie_i, && 1\leq i\leq n.
\end{aligned}\right.\]
\end{thm}

\begin{defn}
An $n$-dimensional Leibniz algebra $L$ is said to be filiform if $
\dim L^i=n-i$ for $2\leq i \leq n$.
\end{defn}

Now let us define a natural graduation for a filiform Leibniz
algebra.

\begin{defn} Given a filiform Leibniz algebra $L$, put $L_i=L^i/L^{i+1}, \ 1 \leq i\leq n-1$, and $\gr(L) = L_1 \oplus L_2\oplus\dots \oplus L_{n-1}$.
 Then $[L_i,L_j]\subseteq L_{i+j}$ and we obtain the graded algebra $\gr(L)$. If $\gr(L)$ and $L$ are isomorphic, then we say that the algebra $L$ is naturally graded.
\end{defn}

Thanks to  \cite{Ayup} and \cite{Ver} it is well known that there
are three types of naturally graded filiform Leibniz algebras. In
fact, the third type encloses the class of naturally graded
filiform Lie algebras.

\begin{thm} \label{T:fil}
 Any complex $n$-dimensional naturally graded filiform Leibniz algebra is isomorphic to one of the following pairwise non-isomorphic algebras:
\begin{align*}
F_n^1: &  \ [e_i,e_1]=e_{i+1}, \  2\leq i \leq {n-1},\\
F_n^2: &  \ [e_1,e_1]=e_{3}, \ [e_i,e_1]=e_{i+1}, \  3\leq i \leq {n-1},\\
F_n^3(\alpha):  &  \
 \left\{\begin{array}{lll}
[e_i,e_1]=-[e_1,e_i]=e_{i+1}, &
2\leq i \leq {n-1},\\[1mm]
[e_i,e_{n+1-i}]=-[e_{n+1-i},e_i]=\alpha (-1)^{i+1}e_n, & 2\leq
i\leq n-1,
\end{array} \right.
\end{align*}
where $\alpha\in\{0,1\}$ for even $n$ and $\alpha=0$ for odd $n$.
\end{thm}

The following theorem decomposes  all $n$-dimensional filiform
Leibniz algebras into  three families of algebras.

\begin{thm}[\cite{Ayup}]\label{th2.5}
 Any complex $n$-dimensional filiform Leibniz algebra admits a basis $\{e_1, e_2, \dots, e_n\}$
 such that the table of multiplication of the algebra has one of the following forms:

$F_1(\alpha_4, \dots, \alpha_n,\theta)=\left\{\begin{array}{ll}
[e_i,e_1]=e_{i+1}, & \  2\leq i \leq {n-1},\\[1mm]
[e_1,e_2]=\theta e_n, & \\[1mm]
[e_j,e_2]=\alpha_4e_{j+2} + \alpha_5e_{j+3}+ \dots +
\alpha_{n+2-j}e_n, & \ 2\leq j \leq {n-2},
\end{array} \right.$ \\[1mm]

$F_2(\beta_4, \dots, \beta_n,\gamma)=\left\{\begin{array}{ll}
[e_1,e_1]=e_{3}, \\[1mm]
[e_i,e_1]=e_{i+1}, & \  3\leq i \leq {n-1},\\[1mm]
[e_1,e_2]=\beta_4e_{4} + \beta_5e_{5}+\dots+ \beta_{n}e_{n}, \\[1mm]
[e_2,e_2]= \gamma e_{n},\\[1mm]
[e_j,e_2]=\beta_4e_{j+2} + \beta_5e_{j+3}+\dots+
\beta_{n+2-j}e_{n}, & \ 3\leq j \leq {n-2},
\end{array} \right.$ \\

$F_3(\theta_1,\theta_2,\theta_3)=
\left\{\begin{array}{lll}
[e_i,e_1]=e_{i+1}, &
2\leq i \leq {n-1},\\[1mm]
[e_1,e_i]=-e_{i+1}, & 3\leq i \leq {n-1}, \\[1mm]
[e_1,e_1]=\theta_1e_n, &   \\[1mm]
[e_1,e_2]=-e_3+\theta_2e_n, & \\[1mm]
[e_2,e_2]=\theta_3e_n, &  \\[1mm]
[e_i,e_j]=-[e_j,e_i] \in \spann<e_{i+j+1}, e_{i+j+2}, \dots , e_n>, &
2\leq i < j \leq {n-1},\\[1mm]
[e_i,e_{n+1-i}]=-[e_{n+1-i},x_i]=\alpha (-1)^{i+1}x_n, & 2\leq
i\leq n-1,
\end{array} \right.$ \\
where  $\alpha\in\{0,1\}$ for even $n$ and $\alpha=0$ for odd $n$.
\end{thm}

Below we present the  description of solvable Leibniz algebras
whose nilradical is isomorphic to the algebra $F_n^1$.

\begin{thm}[\cite{Cas2}] \label{thm33} An arbitrary $(n+1)$-dimensional solvable Leibniz algebra with nilradical $F_n^1$
is isomorphic to one of the following pairwise non-isomorphic algebras:
\[R_1: \left\{\begin{aligned}
{}[e_i,e_1]&=e_{i+1}, && 2\leq i\leq n-1,\\
[x,e_1]& =-e_1-e_2, \\
[e_1,x]& =e_1, \\
[e_i,x]& =(i-1)e_i, && 2\leq i\leq n,
\end{aligned}\right. \qquad R_2(\alpha): \left\{\begin{aligned}
{}  [e_i,e_1]& =e_{i+1}, && 2\leq i\leq n-1,\\
[x,e_1] & =-e_1, \\
[e_1,x]& =e_1,   \\
 [e_i,x] & =(i-1+\alpha)\, e_i, && 2\leq i\leq n,
\end{aligned}\right.\]
\[R_3: \left\{\begin{aligned}
{} [e_i,e_1]& =e_{i+1}, && 2\leq i\leq n-1, & \\
[x,e_1] & =-e_1, \\
[e_1,x]& =e_1,   \\
[e_i,x] & =(i-n)\, e_i, && 2\leq i\leq n, \\
[x,x]&=e_n, \,
\end{aligned}\right.  \qquad  R_4: \left\{\begin{aligned}
{} [e_i,e_1]& =e_{i+1}, && 2\leq i\leq n-1, & \\
[x,e_1] & =-e_1, \\
[e_1,x]& =e_1+e_n,   \\
[e_i,x] & =(i+1-n)\, e_i, && 2\leq i\leq n, \\
[x,x]&=-e_{n-1}, \,
\end{aligned}\right.\]
\[R_5(\alpha_i)=R_5(\alpha_4, \dots, \alpha_{n}):
\left\{\begin{aligned}
{} [e_1,e_1]& =e_3, \\
 [e_i,e_1]&=e_{i+1}, && 2\leq i\leq n-1,\\
 [e_1,x]&=e_2+\sum\limits_{i=4}^{n-1}\alpha_i\, e_i, \\
 [e_i,x]&=e_i+\sum\limits_{j=i+2}^n\alpha_{j-i+2}\, e_j, && 2\leq i\leq n \,.
\end{aligned}\right.\]
Moreover,  the first non-vanishing parameter $\{\alpha_4, \dots,
\alpha_{n}\}$ in the algebra $R_5(\alpha_4, \dots, \alpha_{n})$
can be scaled to 1.
\end{thm}

Due to the work \cite{Camacho} we conclude that there is no
$(n+1)$-dimensional solvable Leibniz algebra whose nilradical is
an algebra from the family $F_1(\alpha_4, \dots, \alpha_n,\theta)$
except the algebra $F_n^1$. Moreover, any $(n+1)$-dimensional
solvable Leibniz algebra whose nilradical is an algebra from the family
$F_3(\theta_1,\theta_2,\theta_3)$ is Lie algebra. Concerning the
second family, from \cite{Camacho} we know that there exists a
solvable Leibniz algebra of dimension $(n+1)$ only when the nilradical
is one of the following:
\begin{align*}
L_1 &= F_{2}(0,0,\dots,0,1), \quad L_{2}^{\beta_{\frac{n+3}{2}}} = F_2(0,0,\dots,0,\beta_{\frac{n+3}{2}},0,\dots,0,1), \quad  n  \ \text{is odd},\\
 L_3^{j} & =F_2^j(0,0,\dots,0,\underbrace{1}_{\beta_j=1},0,\dots,0,0), \quad 4\leq j\leq n.
\end{align*}
In particular, any $(n+1)$-dimensional solvable Leibniz algebra
whose nilradical is either $L_1, \ L_{2}^{\beta_{\frac{n+3}{2}}}$ or $L_3^{j}$
is isomorphic, respectively, to the algebra with the following table of multiplication:
\[R(L_1): \left\{\begin{array}{lll}
[e_1,e_1]=e_3, & & [e_2,e_2]=e_n,\\[1mm]
[e_i,e_1]=e_{i+1}, & 3\leq i\leq n-1,& [x,e_2]=-\frac{n-1}{2}e_2,\\[1mm]
[x,e_1]=-e_1,&&\\[1mm]
[e_1,x]=e_1, && \\[1mm]
[e_2,x]=\frac{n-1}{2}e_2,& &\\[1mm]
[e_i,x]=i e_i,& 3\leq i\leq n,&\\[1mm]
\end{array}\right.\]
\[R(L_{2}^{\beta_{\frac{n+3}{2}}}): \left\{
\begin{array}{llllll}
[e_1,e_1]=e_{3},& &[e_1,e_2]=\beta_{\frac{n+3}{2}} e_{\frac{n+3}{2}},&\\[1mm]
[e_i,e_{1}]=e_{i+1}, & 3\leq i\leq n-1,&[e_2,e_{2}]=e_n,& \\[1mm]
[x,e_1]=-e_1,& &[e_i,e_2]=\beta_{\frac{n+3}{2}}e_{\frac{n-1+2i}{2}},& 3\leq i\leq \frac{n+1}{2},\\[1mm]
[e_1,x]=e_1,& &[x,e_2]=-\frac{n-1}{2} e_2-\beta_{\frac{n+3}{2}} e_{\frac{n+1}{2}},& \\[1mm]
[e_2,x]=\frac{n-1}{2}e_2,&&&\\[1mm]
[e_i,x]=(i-1) e_i, & 3\leq i\leq n,&&\\[1mm]
\end{array}\right.\]
\[R(L_3^{j}): \left\{ \begin{array}{llll}
[e_1,e_1]=e_{3}, &      &                      [e_1,e_2]=e_{j},  &\\[1mm]
[e_i,e_{1}]=e_{i+1}, & 3\leq i\leq n-1, & [e_i,e_2]=e_{j+i-2}, & 3\leq i\leq n+2-j,\\[1mm]
[x,e_1]=-e_1, &        &                      [x,e_2]=-(j-2)e_2-e_{j-1}, &\\[1mm]
[e_1,x]=e_1,&&&\\[1mm]
[e_2,x]=(j-2)e_2,&&&\\[1mm]
[e_i,x]=(i-1) e_i, & 3\leq i\leq n.&&\\[1mm]
\end{array}\right. \]

For acquaintance with the definition of cohomology group of
Leibniz algebras and its applications to the description of the
variety of Leibniz algebras (similar to Lie algebras case) we
refer the reader to the papers
\cite{Balavoine,Gersten,GozKhak,Loday,Loday-Pirashvili,Nijen}.
Here just recall that the second cohomology group of a Leibniz algebra
$L$ with coefficients in a corepresentation $M$ is the  quotient space
\[HL^2(L,M) = ZL^2(L,M)/ BL^2(L,M),\]
 where the  $2$-cocycles $\varphi \in ZL^2(L,M)$ and the $2$-coboundaries $f\in BL^2(L,M)$ are defined
as follows

\begin{equation}\label{E.Z2}(d^2\varphi)(x,y,z)=[x,\varphi(y,z)]
 - [\varphi(x,y), z] + [\varphi(x,z), y] + \varphi(x, [y,z]) - \varphi([x,y],z) + \varphi([x,z],y)=0 \end{equation} and
\begin{equation}\label{E.B2}
 f(x,y) = [d(x),y] + [x,d(y)] - d([x,y]) \  \ \text{for some linear map} \ d. \end{equation}

The following proposition summarizes results regarding derivations
and the second group cohomology for the algebras $RFN_n$ and $R_1$
from the works \cite{Anc} and \cite{KhudOmir2}.
\begin{prop}
\[\begin{array}{ll} \label{pr0}
\dim \Der(RFN_n) = 2, & \qquad  \dim \Der(R_1) = 2,\\[1mm]
\dim BL^2(RFN_n, RFN_n) =(n+1)^2-2, & \qquad  \dim BL^2(R_1, R_1)=(n+1)^2-2,\\[1mm]
\dim ZL^2(RFN_n, RFN_n)=(n+1)^2-2,  & \qquad  \dim ZL^2(R_1, R_1) =(n+1)^2-2,\\[1mm]
\dim HL^2(RFN_n, RFN_n) =0,& \qquad  \dim HL^2(R_1, R_1) =0.\end{array}\]
\end{prop}

\subsection{Degeneration of Leibniz algebras}

The bilinear maps $V \times V \rightarrow V$ form a vector space
$\Hom(V\otimes V, V )$ of dimension $\dim(V)^3$, which can be viewed
with its natural structure of an affine algebraic variety over
the field $\mathbb{F}$. An $n$-dimensional Leibniz algebra $L$ of the
variety $\Leib_n$ may be considered as an element $\lambda(L)$ via the bilinear
mapping $\lambda \colon  L\otimes L  \rightarrow L$ satisfying Leibniz
identity.

The group $\GL_n(F)$ naturally acts on $\Leib_n$ via change of
basis, i.e.,
\[(g*\lambda)(x,y)=g \Big(\lambda \big(g^{-1}(x),g^{-1}(y) \big) \Big), \quad  g \in \GL_n(F), \  \lambda \in \Leib_n.\]

The orbits $\Orb(-)$ under this action are the  isomorphism
classes of algebras.


Note that solvable Leibniz algebras of the same dimension also
form an invariant subvariety of the variety of Leibniz algebras
under the mentioned action.

\begin{defn} It is said that an algebra $\lambda$ degenerates to an algebra $\mu$, if $\Orb(\mu)$ lies in the
Zariski closure of $\Orb(\lambda)$ (denoted by
$\overline{\Orb(\lambda)}$). We denote this by $\lambda
\rightarrow \mu$.
\end{defn}

It is remarkable that $\Orb(NF_n)$ and $\Orb(RNF_n)$ are open sets
in the subvariety of $n$-dimensional nilpotent Leibniz algebras
\cite{Ayup} and the variety of $(n+1)$-dimensional Leibniz
algebras \cite{Anc}, respectively (they are so-called \emph{rigid}
algebras).

In the case when the ground field is the  complex numbers
$\mathbb{C}$, we give an equivalent definition of
degeneration.

\begin{defn} Let $g \colon (0, 1] \rightarrow \GL_n(V)$ be a continuous mapping. We construct a
parameterized family of Leibniz algebras $g_{t} = (V,
[-,-]_{t}), t\in (0, 1]$ isomorphic to $L$. For each $t$ the new
Leibniz bracket $[-,-]_{t}$ on $V$ is defined via the old one as
follows: $[x,y]_{t} = g_{t}[g_{t}^{-1}(x),g_{t}^{-1}(y)]$,
$\forall x, y \in V$. If for any $x,y \in V$ there exists the
limit
\[\lim_{t \to +0}[x,y]_{t} =\lim_{t\to +0}g_{t}[g_{t}^{-1}(x),g_{t}^{-1}(y)] =: [x,y]_0,\]
 then
$[-,-]_0$ is a well-defined Leibniz bracket. The Leibniz algebra
$L_0=(V, [-, -]_0)$ is called a degeneration of the algebra $L$.
\end{defn}



For given Leibniz algebras $\lambda, \mu \in \Leib_{n+1}$ sometimes it is quite difficult to establish the existence of degeneration $\lambda \rightarrow \mu$.
 It is helpful to obtain some necessary invariant conditions for the existence of a degeneration.
 The complete list of the invariants conditions can be found in the works \cite{Burde2}, \cite{Grunewald}, \cite{KOLevelone}, \cite{Seeley}.
  Here we give some of them which we shall use.

We denote by $\Der(\lambda), \ \lambda^m, \Lie(\lambda)$ the space of derivations, powers and maximal Lie subalgebra of the algebra $\lambda$, respectively.

Let $\lambda \rightarrow \mu$ be a nontrivial degeneration. Then the
following inequalities hold:
\begin{equation} \label{eqinvariant}
\dim \Der(\lambda) < \dim \Der(\mu), \quad \dim \lambda^m \geq \dim \mu^m \ \text{for  any} \ m\in \mathbb{N}, \quad \dim \Lie(\lambda) \leq \dim \Lie(\mu).
\end{equation}

Further we shall use $(i, j)$-invariant. This invariant was given
for Lie algebras in \cite{Burde2} and it is also applicable
for  Leibniz algebras.

Let $\lambda \in \Leib_{n+1}$ with structure constants
$(\gamma_{i,j}^k)$, and $(i, j)$ be a pair of positive integers
such that
\[c_{i,j} = \frac {\tr (\ad x)^i \cdot \tr (\ad y)^j} {\tr \big((\ad x)^i \circ (\ad y)^j\big)}\]
 is independent of the
elements $x, y$ of the Leibniz algebra $\lambda$. Then $c_{i,j} (\lambda) = c_{j,i}(\lambda)$ is a
quotient of two polynomials in $C[\gamma_{i,j}^k]$. If neither of
these polynomials is zero, we call $c_{i,j} \in
\mathbb{C}[\gamma_{i,j}^k] $ an $(i, j)$-\textbf{invariant} of
$\lambda$. Suppose $\lambda \in \Leib_{n+1}$ has an $(i,
j)$-invariant $c_{i,j}$.  Then all $\mu \in
\overline{\Orb(\lambda)} $ have the same $(i, j)$-invariant.

Denote by $\mathcal{LR}_n(N)$ the set of all $n$-dimensional
solvable Leibniz algebras whose nilradical is $N$. In the paper
\cite{Casas} it is proved that a given degeneration between two
solvable Leibniz algebras implies some restriction on their
nilradicals. In particular, in the case of equality of dimensions
of nilradicals the existence of degeneration between solvable
Leibniz algebras implies the existence of degeneration between their
nilradicals.

\begin{prop} \label{P:dim} Let $R_1, R_2$ be $n$-dimensional solvable Leibniz algebras and let $R_1 \in \mathcal{L}R_n(N_1), \
R_2 \in \mathcal{L}R_n(N_2)$. If $R_1 \rightarrow R_2$, then $\dim
N_2 \geq \dim N_1$.  Moreover, if $\dim N_1 = \dim N_2$ and $R_1
\rightarrow R_2$, then $N_1 \rightarrow N_2$.
\end{prop}

\section{Some irreducible components of the variety $\Leib_{n+1}$}\label{S:irred}

In this section we present some irreducible components of the
variety $\Leib_{n+1}$ in terms of closures of orbits of some
Leibniz algebras. The following equality $\dim \Orb(\lambda) =
(n+1)^2 - \dim \Der(\lambda)$ gives us the dimensional relation
between orbits and derivations of an algebra $\lambda$.

Let $L$ be an $(n+1)$-dimensional solvable Leibniz algebra whose
nilradical is the filiform algebra $F_n^1$. The proposition below
describes derivations of the algebras from $R_2(\alpha) -
R_5(\alpha_i)$ in the list of Theorem~\ref{thm33}.

\begin{prop} \label{pr1} Derivations of the algebras $R_2(\alpha) - R_5(\alpha_i)$ have the following form:

\[\begin{array}{cl}
\Der(R_2(\alpha)): & \left\{\begin{array}{lll}
                d_1(e_1)=e_1, & d_1(e_i)=(i-2)e_i,& 3\leq i\leq n,\\
                d_2(e_i)=e_i, & & 2\leq i\leq n,\\
                d_3(x)=-e_1, &  d_3(e_i)=e_{i+1},& 2\leq i\leq n-1;
\end{array}\right. \quad \alpha \neq 2-n; 1-n,\\[6mm]
\Der(R_2(2-n)): & \left\{\begin{array}{lll}
                d_1(e_1)=e_1, & d_1(e_i)=(i-2)e_i,& 3\leq i\leq n,\\
                d_2(e_i)=e_i, & & 2\leq i\leq n,\\
                d_3(x)=-e_1, &  d_3(e_i)=e_{i+1},& 2\leq i\leq
                n-1,\\
                d_4(x)= - e_{n-1}, &  d_4(e_1)=e_{n};
\end{array}\right. \\[6mm]
\Der(R_2(1-n)): & \left\{\begin{array}{lll}
                d_1(e_1)=e_1, & d_1(e_i)=(i-2)e_i,& 3\leq i\leq n,\\
                d_2(e_i)=e_i, & & 2\leq i\leq n,\\
                d_3(x)=-e_1, &  d_3(e_i)=e_{i+1},& 2\leq i\leq
                n-1,\\
                d_4(x)= e_{n}; \end{array}\right. \\[6mm]
                \end{array}\]
\[\begin{array}{cl}
\Der(R_3):   & \left\{\begin{array}{lll}
                d_1(e_1)=e_1, & d_1(e_i)=(i-n)e_i, & 2\leq i\leq n-1,\\
                d_2(x)=e_1, & d_2(e_i)=-e_{i+1},& 2\leq i\leq n-1,\, \\
                d_3(x)=e_n;
\end{array}\right.\\[6mm]
\Der(R_4):    & \left\{\begin{array}{lll}
                d_1(e_1)=e_1, & d_1(e_i)=(i+1-n)e_i,& 2\leq i\leq n,\\
             d_2(x)=-e_1,  &  d_2(e_i)=e_{i+1},& 2\leq i\leq n-1,\\
              d_3(e_1)=e_n, & d_3(x)=-e_{n-1};
\end{array}\right.\\[6mm]
\end{array}\]
\[\begin{array}{cl}
\Der(R_5(0)): &  \left\{\begin{array}{lll}
                 d_1(e_1)=e_1, & d_1(e_i)=(i-1)e_i, & 2\leq i\leq n,\\
                 d_2(e_1)=e_2, & d_2(e_i)=e_i, & 2\leq i\leq n,\\
                 d_j(e_1)=e_j, & d_j(e_i)=e_{i+j-2}, & 3\leq j\leq n, \ 2\leq i\leq n-j+2;
\end{array}\right.\\[6mm]
\Der(R_5(\alpha_i)): &  \left\{\begin{array}{llll}
                 d_1(e_1)=e_2, & d_2(e_i)=e_2 + \alpha_ne_n, &  d_1(e_i)=e_i, & 3\leq i\leq n,\\
                 d_j(e_1)=e_{j+1}, & d_j(e_i)=e_{i+j-1}, & 2\leq j\leq n-1, & 2\leq i\leq n-j+1,
\end{array}\right.
\end{array}\]
where in the case of $R_5(\alpha_i)$ one of the parameters
$\alpha_i$ is nonzero.
\end{prop}
\begin{proof} The proof of the proposition is carrying out by direct computations.
\end{proof}

Due to equality \eqref{E.B2} defining the space $BL^2$ we have
\begin{cor}
\begin{align*}
\dim BL^2\big(R_2(\alpha), R_2(\alpha)\big) &=
\left\{\begin{array}{ll}(n+1)^2 - 4, & \alpha = 2-n \ \text{or} \ 1-n, \\
(n+1)^2 - 3, & \alpha \neq 2-n, 1-n; \end{array}\right.\\
\dim BL^2(R_3, R_3) &= (n+1)^2 - 3; \\
\dim BL^2(R_4, R_4)&= (n+1)^2 - 3; \\
\dim BL^2\big(R_5(\alpha_i), R_5(\alpha_i)\big) &= \left\{\begin{array}{ll}n^2+n+1, & \alpha_i = 0 \ \ \text{for all} \ i,   \\
n^2+n+2, & \alpha_i \neq 0 \ \  \text{for  some}   \ i.
\end{array}\right.
\end{align*}
\end{cor}

The following result presents values of $(1,1)$-invariants for the algebras $RNF_n$, $ R_1-R_5(\alpha_i)$, $ R(L_{1}), \ R(L_{2}^{\beta_{\frac{n+3}{2}}}),  \ R(L_{3}^j)$.

\begin{prop} \label{prop2}
\begin{align*}
c_{1,1}(RNF_n)&= \frac{(1+2+3+\dots+n)^2}{1+2^2+3^2+\dots+n^2}= \frac{3n(n+1)} {2(2n+1)},\\
c_{1,1}(R_1)&= \frac{(1+1+2+3+\dots+n-1)^2}{1+1+2^2+3^2+\dots+(n-1)^2}=\frac{3(n^2-n+2)^2}{2(6+n(n-1)(2n-1))},\\
c_{1,1}(R(L_{1}))&= \frac{(1+2+3+\dots+n-1+\frac{n-1}{2})^2}{1+2^2+3^2+\dots+(n-1)^2+(\frac{n-1}{2})^2} =\frac{3(n^2-1)}{4n-3},\\
c_{1,1}(R(L_{2}^{\beta_{\frac{n+3}{2}}}))&=
\frac{(1+2+3+\dots+n-1+\frac{n-1}{2})^2}{1+2^2+3^2+\dots+(n-1)^2+(\frac{n-1}{2})^2}=\frac{3(n^2-1)}{4n-3},\\
c_{1,1}(R(L_{3}^j))&=\frac{(1+2+3+\dots+n-1+j-2)^2}{1+2^2+3^2+\dots+(n-1)^2+(j-2)^2} =
\frac{3(n^2 - n +2j-4)^2}{4(n^3-2n^2+n+3(j-2)^2)},\\
c_{1,1}(R_2(\alpha)) &= \frac{(1+(1+\alpha)+(2+\alpha)+\dots+(n-1+\alpha))^2}
{1^2+(1+\alpha)^2+(2+\alpha)^2+\dots+(n-1+\alpha)^2},\\
c_{1,1}(R_3) &=\frac{(1+(2-n)+(3-n)+\dots+(-1))^2}{1^2+(2-n)^2+(3-n)^2+\dots+1^2}=
\frac{3n(n-3)^2}{2(2n^2-9n+15)},\\
c_{1,1}(R_4) &= \frac{(1+(3-n)+(4-n)+\dots+(-1)+0+1)^2}{1^2+(3-n)^2+(4-n)^2+\dots+1^2+0^2+1^2}=
\frac{3(n^2-5n+10)^2}{2(2n^3-15n^2+37n-18)},\\
c_{1,1}(R_5(\alpha_{i})) &= n-1.
\end{align*}
\end{prop}
\begin{proof} The proof of the proposition is carrying out by direct calculations.
\end{proof}

\begin{rem} \label{rem1}
 From Proposition~\ref{prop2} it is easy to see that
\begin{align*}
c_{1,1}(R_2(1))  & =c_{1,1}(RNF_n), \qquad c_{1,1}(R_2(0)) =c_{1,1}(R_1),\\
c_{1,1}(R_2(1-n))  &  =c_{1,1}(R_3), \qquad  c_{1,1}(R_2(2-n))  =c_{1,1}(R_4).
\end{align*}
\end{rem}

Below, we give some degenerations.

\begin{exam} \label{exam} There exist the  following degenerations
\[\begin{array} {llllll}
F_n^1\rightarrow F_n^2 & \text{via} & g_t(e_1) = e_1 - t^{-1}e_2, & g_t(e_2) =
t^{-1}e_2, &  g_t(e_i) = e_i, & 3 \leq i \leq n;\\[1mm]
RNF_n\rightarrow R_2(1) & \text{via} & g_t(x) = x, & g_t(e_1) = t^{-1}e_1,& g_t(e_i) = t^{-i+2}e_i, & 2 \leq i \leq n;\\[1mm]
R_1 \rightarrow R_2(0) & \text{via} & g_t(x) = x, & g_t(e_1) = e_1,  & g_t(e_i) = t e_i, & 2 \leq i \leq n;\\[1mm]
R_3 \rightarrow R_2(1-n) & \text{via} & g_t(x) = x, & g_t(e_1) = e_1, & g_t(e_i) = t e_i, & 2 \leq i \leq n;\\[1mm]
R_4 \rightarrow R_2(2-n) & \text{via} & g_t(x) = x, & g_t(e_1) = e_1, & g_t(e_i) = t e_i, & 2 \leq i \leq n;\\[1mm]
R_5(\alpha_i) \rightarrow R_5(0)& \text{via} & g_t(x) = x, &
g_t(e_1) = te_1, & g_t(e_i) = t^{i-1} e_i, & 2 \leq i \leq
n.\end{array}\]
\end{exam}

Now we present representatives of irreducible components of the variety $\Leib_{n+1}$.

\begin{thm} \label{thm111} $\overline{\Orb(R_3)}, \ \overline{\Orb(R_4)}$ and $\overline{\cup_{\alpha\in K}\Orb(R_2(\alpha))}$,
where $K=\mathbb{C}\setminus \{0, 1, 1-n, 2-n\}$, are irreducible components of $\Leib_{n+1}$.
\end{thm}
\begin{proof} Firstly, we will prove the assertion of the theorem  for the algebra $R_3$, that is, we will prove the non-existence of degeneration
 from any algebra $X\in \Leib_{n+1}$ to the algebra $R_3$. Let us assume the contrary, i.e., $X\rightarrow R_3$, then from Proposition~\ref{pr1}
  and inequalities \eqref{eqinvariant} we conclude that $\dim(\Der X) < \dim(\Der R_3)=3$.
 Actually, the algebra $X$ is a solvable Leibniz algebra. Indeed, if $X$ is not solvable then by Levi's Theorem \cite{Bar} it  decomposes
  into a semidirect sum of a semi-simple Lie algebra $S$ and a solvable radical $\Rad(X)$. Since in a semi-simple Lie algebra
  (which, clearly, has dimension greater or equal to 3) an operator $\ad(x)$ for any $x\in S$ is a derivation of the algebra $X$,
   then we obtain $\dim(\Der X) \geq 3$. Therefore, the algebra $X$ is solvable.

Taking into account that $\dim R_3^2 = n$ and solvability of $X$ from inequalities \eqref{eqinvariant}, we derive that $\dim X^2 = n$.
 Since the square of a solvable algebra belongs to the nilradical, then the solvable algebra $X$ has nilradical of dimension $n$.
  Therefore, the nilradical of $X$ degenerates to $F_n^1$ (because nilradical of the algebra $R_3$ is $F_n^1$).
   Again applying inequality \eqref{eqinvariant} we derive that the nilradical of $X$ is one of the following algebras:
\[NF_n, \qquad F_n^1(\alpha_4, \alpha_5, \dots, \alpha_n, \theta), \qquad F_2(\beta_4, \dots, \beta_n,\gamma), \qquad F_3(\theta_1,\theta_2,\theta_3).\]

The non-trivial degeneration $F_n^1 \rightarrow F_n^2$ from Example~\ref{exam} implies the non-existence of degeneration from the algebra $F_n^2$ to the algebra $F_n^1$.

From arguments above we conclude that the possibilities for the solvable algebra $X$ are the following:
\[RNF_n, \qquad R_1 - R_5(\alpha_i), \qquad R(L_{1}), \qquad R(L_{2}^{\beta_{\frac{n+3}{2}}}), \qquad R(L_{3}^j).\]

Comparing $c_{1,1}$ invariants from Proposition~\ref{prop2}, dimensions of the spaces of derivations from Propositions~\ref{pr0}, \ref{pr1} and applying inequalities \eqref{eqinvariant},
 we conclude that none of the above algebras degenerates to the algebra $R_3$. Hence, $\overline{\Orb(R_3)}$ is an irreducible component of the variety $\Leib_{n+1}$.

 The assertion of theorem for the algebra $R_4$ and the family of algebras $R_2(\alpha)$, with $\alpha\in K$,
  is proved applying the same arguments as used for the algebra $R_3$ and degenerations from Example~\ref{exam}.
\end{proof}

Let us assume that $X\rightarrow R_5(\alpha_4, \dots, \alpha_n)$. Since $\dim \Der(R_5(\alpha_4, \dots, \alpha_n))>3 $
 then from dimensions arguments as they were used for the algebra $R_3$ we cannot assert the solvability of $X$.
  Nevertheless applying similar arguments as in Theorem~\ref{thm111} we obtain the following result.
\begin{prop}
 $\overline{\cup_{\alpha\in \mathbb{C^*}}\Orb(R_5(\alpha_4, \dots, \alpha_n))}$ forms an irreducible component
 in the subvariety of $(n+1)$-dimensional solvable Leibniz algebras of the variety $\Leib_{n+1}$.
\end{prop}

\subsection{Second cohomology of the algebras $R_2-R_5$}

In this subsection we give  descriptions of the second cohomology group
of the algebras $R_2 - R_5$ by presenting their bases. In fact, we find bases of the space $BL^2$ and $ZL^2$ for these algebras.
 These descriptions can be applied in the study of infinitesimal deformations and extensions of mentioned algebras (see works \cite{Fial2,Fial3,Khak,KhudOmir,Mil,Nijen} and references therein).

In the next theorem we present the general form of $2$-cocycles for the algebra $R_3$.

\begin{thm} \label{thR_3} An arbitrary $\varphi\in ZL^2(R_3, R_3)$ has the following form:
\begin{align*}
\varphi(e_1,e_1) & = \sum\limits_{i=3}^{n}a_{1,i}e_i,\\
\varphi(e_i,e_1) & = a_{i,0}x+ \sum\limits_{s=1}^{n}a_{i,s}e_s, \quad 2\leq i\leq n-1,\\
\varphi(e_n,e_1) & = a_{n-1,0}e_1-\sum\limits_{i=3}^{n-1}(\sum\limits_{t=2}^{i-1}a_{n+t-i,t})e_i+a_{n,n}e_n,\\
\varphi(e_1,e_2) & = b_{1,1}e_1,\\
\varphi(e_i,e_2) & = (i-n)b_{1,1}e_i+b_{2,3}e_{i+1}, \quad 2\leq i\leq n,\\
\varphi(e_1,e_i) & = -a_{i-1,0}e_1, \ 3\leq i\leq n,\\
\varphi(e_i,e_3) & = (n-i)a_{2,0}e_i-(a_{2,1}+b_{1,1})e_{i+1}, \quad 2\leq i\leq n-1,\\
\varphi(e_i,e_j) & = (n-i)a_{j-1,0}e_i+(a_{j-2,0}-a_{j-1,1})e_{i+1}, \ 2\leq i\leq n-1, \quad 4\leq j\leq n,\\
\varphi(e_1,x) & = \frac{2}{(n-2)(n-1)}\sum\limits_{t=2}^{n}a_{t,t} x+c_{1,1}e_1+\sum\limits_{i=2}^{n-1}(i-n-1)a_{1,i+1}e_i+c_{1,n}e_n,\\
\varphi(e_2,x) & = (n-2)b_{1,1}x+ (n-1)b_{2,3}e_1+\sum\limits_{i=2}^{n}c_{2,i}e_i, \displaybreak \\
\varphi(e_3,x) & =(3-n)a_{2,0}x+ (2-n)(a_{2,1}+b_{1,1})e_1+(a_{2,2}-\frac{2}{n-1}\sum\limits_{t=2}^{n}a_{t,t})e_2+(c_{2,2}+c_{1,1})e_3 \\
{} & + \sum\limits_{i=4}^{n-1}(c_{2,i-1}+(3-i)a_{2,i})e_i+(c_{2,n-1}+(3-n)a_{2,n}-a_{2,0})e_n,\\
\varphi(e_i,x) & =(i-n)a_{i-1,0}x +(n-i+1)(a_{i-2,0}-a_{i-1,1})e_1+\sum\limits_{s=2}^{i-2}(i-s)\sum\limits_{t=2}^{s}a_{i+t-s-1,t}e_s \\
{} & +(\sum\limits_{t=2}^{i-1}a_{t,t}+ \frac{(i+1-2n)(i-2)}{(n-2)(n-1)}\sum\limits_{t=2}^{n}a_{t,t})e_{i-1} +(c_{2,2}+(i-2)c_{1,1})e_i \\
{} & + \sum\limits_{s=i+1}^{n-1}(c_{2,s-i+2}+(i-s)\sum\limits_{t=2}^{i-1}a_{t,s-i+t+1})e_s \\
{} &+(c_{2,n-i+2}-a_{i-1,0}+(i-n)\sum\limits_{t=2}^{i-1}a_{t,n-i+t+1})e_n, \quad 4\leq i\leq n, \\
 \varphi(x,e_1) & = -\frac{2}{(n-2)(n-1)}\sum\limits_{t=2}^{n}a_{t,t}x - c_{1,1}e_1+a_{1,3}e_2+ \sum\limits_{i=3}^{n}d_{1,i}e_i,\\
\varphi(x,e_2) &= -b_{2,3}e_1+b_{1,1}e_n,\\
\varphi(x,e_3) &= (a_{2,1}+b_{1,1})e_1-a_{2,0}e_n,\\
\varphi(x,e_i) &= (a_{i-1,1}-a_{i-2,0})e_1-a_{i-1,0}e_n, \quad  4\leq i\leq n,\\
\varphi(x,x)& = a_{n-1,0}x + (a_{n-1,1}-a_{n-2,0})e_1 + \sum\limits_{i=2}^{n-2}\big((n-i)(a_{1,i+2}-d_{1,i+1})
+ \sum\limits_{s=2}^{i}a_{n-i+s-1,s}\big)e_i \\
{} & -(2c_{1,0}+d_{1,n}+c_{1,n}+a_{n,n})e_{n-1}+\beta_ne_n.
\end{align*}
\end{thm}
\begin{proof}
We set
\[\left\{\begin{array}{ll}
\varphi(e_i,e_1) = a_{i,0}x +\sum\limits_{s=1}^{n}a_{i,s}e_s, \
1\leq i\leq n, &
\varphi(e_1,x) = c_{1,0}x+ \sum\limits_{i=1}^{n}c_{1,i}e_i,\\[1mm]
\varphi(e_1,e_2) = b_{1,0}x+\sum\limits_{i=1}^{n}b_{1,i}e_i, &
\varphi(e_2,x) = c_{2,0}x+\sum\limits_{i=1}^{n}c_{2,i}e_i, \\[1mm]
\varphi(e_2,e_2) = b_{2,0}x+\sum\limits_{i=1}^{n}b_{2,i}e_i, &
\varphi(x,e_1) = d_{1,0}x + \sum\limits_{i=1}^{n}d_{1,i}e_i, \\[1mm]
\varphi(x,x) = \beta_{0}x \sum\limits_{i=1}^{n}\beta_{i}e_i, &
\varphi(x,e_2) =d_{2,0}x+ \sum\limits_{i=1}^{n}d_{2,i}e_i.
\end{array} \right.\]

Applying equality \eqref{E.Z2} for the triple $(e_i, e_1, e_1)$, we obtain $[e_i,\varphi(e_1,e_1)]=0$.
 Hence, $\varphi(e_1,e_1) = \sum\limits_{i=2}^{n}a_{1,i}e_i$. Similarly,
  the equation $(d^2\varphi)(e_i, e_j, e_k)=0$, for $2\leq j,k\leq n$, leads to $[e_i,\varphi(e_j,e_k)]=0$. Consequently, we
have  $\varphi(e_j,e_k) \in \spann\left<e_2, e_3,\dots ,e_n\right>$.

Considering $(d^2\varphi)(e_1, e_1, e_2)=0$, we derive
\[\varphi(e_1,e_2) = b_{1,0}x + b_{1,1}e_1+b_{1,n}e_n.\]

Moreover, from $(d^2\varphi)(e_i, e_1, e_2)=0$ with $2 \leq i \leq
n-1$, we deduce
\[\varphi(e_{i+1},e_j)= [e_i,\varphi(e_1,e_2)]+[\varphi(e_i,e_2),e_1],\]
which inductively get
\[\varphi(e_i,e_2) = b_{1,0} \frac {(i-2)(i+1-2n)} {2}e_i+
\big((i-2)b_{1,1}+b_{2,2}\big)e_i+\sum\limits_{s=i+1}^{n}b_{2,s-i+2}e_s,
\quad  2\leq i\leq n.\]

Similarly, we have
\[\begin{array}{lllll}(d^2\varphi)(e_n, e_1, e_2)=0& \Rightarrow &  b_{1,0}
=0,\\[1mm]
(d^2\varphi)(e_1, e_2, x)=0 & \Rightarrow & b_{1,n} =0, & c_{2,0}
= (n-2)b_{1,1},\\[1mm] (d^2\varphi)(e_2, e_2, x)=0& \Rightarrow &
 b_{2,2} = (2-n)b_{1,1}, & c_{2,1} = (n-1)b_{1,3}, & b_{2,i}=0,
\ 4 \leq i \leq n.
\end{array}\]

Now we consider  $(d^2\varphi)(e_i, e_j, e_1)=0$ with $2\leq j\leq
n-1$. Then we have
\[\varphi(e_i,e_{j+1}) = [\varphi(e_i,e_j),e_1] - [e_i,\varphi(e_j,e_1)]- \varphi([e_i,e_1],e_j).\]

Applying the induction on $j$  for any $i$ and the equality above we obtain the following:
\[\begin{array}{ll}\varphi(e_1,e_j) =-a_{j-1,0}e_1, & 3\leq j\leq n, \\[1mm]
\varphi(e_i,e_3) = (n-i)a_{2,0}e_i-(a_{2,1}+b_{1,1})e_{i+1},  &
2\leq i\leq n-1,\\[1mm]
\varphi(e_i,e_j) = (n-i)a_{j-1,0}e_i+(a_{j-2,0}-a_{j-1,1})e_{i+1},
& 2\leq i\leq n-1, \ 4\leq j\leq n. \end{array}\]

On the other hand, the condition $(d^2\varphi)(e_i, e_n, e_1)=0$
implies \[a_{n, 0} =0, \quad a_{n,1} = a_{n-1,0}.\]

We consider equality \eqref{E.Z2} for the triple $(e_1, x, e_1)$, then we get \[d_{1,0} =-c_{1,0}, \quad
a_{1,2} =0, \quad c_{1,i} = (i-1-n)a_{1, i+1}, \ 2 \leq i \leq
n-1.\]

Thus, we have
\begin{align*}
\varphi(e_1,x) & =c_{1,0}x + c_{1,1}e_1+\sum\limits_{i=2}^{n-1}(i-n-1)a_{1,i+1}e_i+c_{1,n}e_n,\\
\varphi(e_2,x) & = (n-2)b_{1,1}x + (n-1)b_{2,3}e_1 + \sum\limits_{i=2}^{n}c_{2,i}e_i
\end{align*}
 and
 $[e_1,\varphi(e_i,x)]=(i-n)a_{i-1,0}e_1$ for $ 3\leq i\leq n$.

From the equality $(d^2\varphi)(e_i, e_1, x)=0, \ 2\leq i\leq n-1$,
we have
\begin{align*}
\varphi(e_{i+1},x) & = [\varphi(e_i,x),e_1] +  [e_i,\varphi(e_1,x)]
- [\varphi(e_i,e_1), x] + \varphi(e_i, [e_1,x]) + \varphi(
[e_i,x], e_1)\\
{} & = [\varphi(e_i,x),e_1] +  (i+1-n)x + c_{1,0}e_i+c_{1,1}e_{i+1} + (i-n)a_{i,1}e_1 + \sum\limits_{s=2}^n
(i+1-s)a_{i,s}e_s.
\end{align*}

Hence, we obtain inductively that
\[\begin{array}{ll}
\varphi(e_3,x) &
=(3-n)a_{2,0}x+
(2-n)(a_{2,1}+b_{1,1})e_1+(a_{2,2}+(2-n)c_{1,0})e_2+(c_{2,2}+c_{1,1})e_3
\\[1mm] &
+\sum\limits_{i=4}^{n-1}(c_{2,i-1}+(3-i)a_{2,i})e_i+(c_{2,n-1}+(3-n)a_{2,n}-a_{2,0})e_n,\\[1mm]
\varphi(e_i,x) &
= (i-n)a_{i-1,0}x + (n-i+1)(a_{i-2,0}-a_{i-1,1})e_1+\sum\limits_{s=2}^{i-2}(i-s)\sum\limits_{t=2}^{s}a_{i+t-s-1,t}e_s\\
&+\big(\frac{(i+1-2n)(i-2)}{2}c_{1,0}+\sum\limits_{t=2}^{i-1}a_{t,t}\big)e_{i-1}+(c_{2,2}+(i-2)c_{1,1})e_i
\\
&+\sum\limits_{s=i+1}^{n-1}(c_{2,s-i+2}+(i-s)\sum\limits_{t=2}^{i-1}a_{t,s-i+t+1})e_s+
(c_{2,n-i+2}-a_{i-1,0}+(i-n)\sum\limits_{t=2}^{i-1}a_{t,n-i+t+1})e_n,
\end{array}\]
where $4 \leq i \leq n$.

Moreover, the condition $(d^2\varphi)(e_n, e_1, x)=0$ implies
$[\varphi(e_n,x),e_1] - [\varphi(e_n,e_1), x] + \varphi(e_n,
[e_1,x])=0$, which derives
\[(n-1)a_{n,2}e_2 + \sum\limits_{s=3}^{n-1} (n-s+1)\big(a_{n,s}+\sum\limits_{t=2}^{s-1}a_{n+t-s,t}\big)e_s
+
\Big(-\frac{(n-1)(n-2)}{2}c_{1,0}+\sum\limits_{t=2}^{n}a_{t,t}\Big)
e_n=0.\]

Thus, we get
\[a_{n,2} =0,  \quad c_{1,0} = \frac{2}{(n-1)(n-2)}
\sum\limits_{t=2}^{n}a_{t,t}, \quad a_{n,s} = -
\sum\limits_{t=2}^{s-1}a_{n+t-s,t}, \ 3 \leq s \leq n-1.\]

Considering equality \eqref{E.Z2} for the
following triples $(e_2, x, e_1), \ (e_1, x, e_2), \ (x, e_1, e_2), \ (e_2, x, e_2)$  we obtain:
\[d_{1,1} = -c_{1,1}, \quad d_{2,0} =0, \quad d_{2,s}=0, \ 2 \leq s \leq n-1, \quad d_{2,1} = -b_{2,3}.\]

From $(d^2\varphi)(x, e_i, e_1)=0$ with $2\leq i\leq n-1,$ we have
\[\varphi(x,e_{i+1}) = [\varphi(x, e_i),e_1] - [x,\varphi(e_i,e_1)] + \varphi(e_1, e_i),\]
which inductively implies
\begin{align*}
\varphi(x,e_3) & =(a_{2,1}+b_{1,1})e_1-a_{2,0}e_n,\\
\varphi(x,e_i) & =(a_{i-1,1}-a_{i-2,0})e_1-a_{i-1,0}e_n, \qquad 4\leq i\leq n.
\end{align*}

Finally, the equalities $(d^2\varphi)(e_1, x, x)=(d^2\varphi)(x, e_1, x)=(d^2\varphi)(x, e_2, x)=0$ imply
\begin{align*}
d_{1,2} & = a_{1,3}, \qquad  \qquad  \  d_{2,n} = b_{1,1},\\
\beta_0 & = a_{n-1,0}, \qquad  \qquad \beta_1 = a_{n-1,1} - a_{n-2,0},\\
\beta_{n-1} & = -d_{1,n}-c_{1,n}-a_{n,n} -\frac{4}{(n-1)(n-2)}
\sum\limits_{t=2}^{n}a_{t,t},\\
\beta_i & = (n-i)(a_{1,i+2}-d_{1,i+i}) + \sum\limits_{s=2}^i a_{n-i+s-1}, \quad 2 \leq i \leq n-2,
\end{align*}
which complete the proof of the theorem.
\end{proof}

\begin{cor} $\dim ZL^2(R_3, R_3) = (n+1)^2 - 2 $ and $\dim HL^2(R_3, R_3) = 1$.
\end{cor}

Now we present a basis of  $HL^2(R_3, R_3)$.

\begin{prop}\label{prop3.10}
The adjoint class $\overline{\xi}$ forms a basis of $HL^2(R_3,
R_3)$, where
\[\xi: \left\{\begin{array}{ll}
\xi(e_1,x) = e_1, \\[1mm]
\xi(e_i,x) = (i-2)e_i,& 3\leq i\leq n, \\[1mm]
\xi(x,e_1) = -e_1.
\end{array} \right.\]
\end{prop}

\begin{proof}
In order to find a basis of $HL^2(R_3,R_3)$ we need to describe linear independent elements which lie in $ZL^2(R_3, R_3)$ and do not lie in $BL^2(R_3, R_3)$.
 For achieve this purpose we will find a basis of $2$-cocycles and $2$-coboundaries.

Since an arbitrary element of $ZL^2(R_3, R_3)$ has the form of Theorem~\ref{thR_3} we shall use this description.

Note that there are parameters $(a_{i,j}, b_{1,1}, b_{2,3}, c_{1,1}, c_{1,n}, \beta, c_{2,k}, d_{1,s})$ in the general form of elements $ZL^2(R_3, R_3)$.
 One of the natural basis of the space $ZL^2$ is a basis whose basis elements are obtained by the instrumentality of these parameters. For the fixed pair $(i,j)$ we
denote by $\varphi(a_{i,j})$ a cocycle which has $a_{i,j} =1$ and all other  parameters  are equal to zero. Define such type of notation for other parameters.

Set
\begin{align*}
\varphi_{i,j} & = \varphi(a_{i,j}), \quad
\psi_{1}= \varphi(b_{1,1}), \quad \psi_{2}= \varphi(b_{2,3}),
\quad \psi_{3}= \varphi(c_{1,1}), \\
\psi_{4} & = \varphi(c_{1,n}), \quad \psi_{5}= \varphi(\beta),\qquad \eta_{k}=
\varphi(c_{2,k}),
 \quad \rho_{s}= \varphi(d_{1,s}),
 \end{align*}
where $1\leq i \leq n $, $0\leq j \leq n$, $ 2\leq k \leq n $,
$3\leq s \leq n$, and
\[(i,j) \notin \{(1,0), (1,1), (1,2), (n,0),(n,1), \dots, (n, n-1)\}.\]

In order to find a basis of $BL^2(R_3, R_3)$ we consider the
endomorphisms $f_{j,k} :R_3 \rightarrow R_3$ defined as follow
\begin{align*}
f_{i,j}(e_i) & =e_j,  \qquad 1 \leq i, j \leq n,\\
f_{i,n+1}(e_i) & =x, \qquad  \ 1 \leq i \leq n,\\
f_{n+1,j}(x) & = e_j, \qquad  1 \leq j \leq n,\\
f_{n+1,n+1}(x) &  = x,
\end{align*}
where in the expansion of endomorphisms the omitted values are assumed to be zero.

Consider  \[g_{i,j}(x,y) = [f_{i,j}(x),y] + [x,f_{i,j}(y)] -
f_{i,j}([x,y]).\]

Note that $g_{i,j} \in BL^2(R_3, R_3)$ and now we separate a basis
from these elements. Since the dimension of the space $\Der(R_3)$ is equal to 3, then to take a
basis we should exclude three elements $g_{i,j}$. The description
of $\Der(R_3)$ allow  us to release the elements $g_{1,1},g_{2,3}$
and $g_{n+1,n}$.

By direct computation we obtain
\[g_{j,k} : \left\{\begin{array}{ll}
g_{1,2} = -\varphi_{1,3},\\[1mm]
g_{1,i} = -\rho_i-\varphi_{1,i+1}, & 3\leq i\leq n-1,\\[1mm]
g_{1,n} = \psi_4 -\rho_n,\\[1mm]
g_{1,n+1} = \sum\limits_{k=2}^{n-1}(n-k)\varphi_{k,k}-\psi_4-\rho_n,\\[1mm]
g_{2,1} = -\psi_2,\\[1mm]
g_{2,i} = -\varphi_{2,i+1}+(2-i)\eta_i, & 2\leq i\leq n-1, \ i\neq 3,\\[1mm]
g_{2,n} = (2-n)\eta_n,\\[1mm]
g_{2,n+1} = \varphi_{2,1}-\psi_1-\eta_n,   \\[1mm]
g_{i,k} = \varphi_{i-1,k}-\varphi_{i,k+1}, & 3\leq i\leq n-1, \ 1\leq k\leq n-1, \\[1mm]
g_{i,n} = \varphi_{i-1,n}, & 3\leq i\leq n-1,  \\[1mm]
g_{i,n+1} = \varphi_{i-1,0}+\varphi_{i,1}, & 3\leq i\leq n-1,\\[1mm]
g_{n,1} =\varphi_{n-1,1},  \\[1mm]
g_{n,i} = \varphi_{n-1,i}, &  2\leq i\leq n-2,  \\[1mm]
g_{n,n-1} = \varphi_{n-1,n-1}-\varphi_{n,n},  \\[1mm]
g_{n,n} = \varphi_{n-1,n}+\psi_5,\\[1mm]
g_{n,n+1} = \varphi_{n-1,0}, \\[1mm]
g_{n+1,1} =-\eta_3,  \\[1mm]
g_{n+1,i} = -\rho_{i+1}, &  2\leq i\leq n-1,     \\[1mm]
g_{n+1,n+1} =\psi_3-2\psi_5+(n-2)\eta_2.
\end{array} \right.\]

From these equalities, it is not difficult to check that $\psi_3$
and $\eta_2$ do not belong to $BL^2(R_3, R_3)$, but $\psi_3
+(n-2)\eta_2 \in BL^2(R_3, R_3)$. Thus, we can take the adjoint class of $\psi_3$ as a basis of
$ZL^2(R_3, R_3)$.
\end{proof}

\begin{thm}\label{thmR_4}
An arbitrary $\varphi\in ZL^2(R_4, R_4)$ has the following form:
\begin{align*}
\varphi(e_1,e_1) &= \sum\limits_{i=3}^{n}a_{1,i}e_i,\\
\varphi(e_i,e_1) &= a_{i,0}x+ \sum\limits_{s=1}^{n}a_{i,s}e_s, \ 2\leq i\leq n-1,\\
\varphi(e_n,e_1) & = a_{n-1,0}e_1-\sum\limits_{i=3}^{n-1}(\sum\limits_{t=2}^{i-1}a_{n+t-i,t})e_i-(\sum\limits_{t=2}^{n-1}a_{t,t}+\frac{(n-4)(n-1)}{2}d_{1,0})e_n,\\
 \varphi(e_1,e_2) & = b_{1,1}e_1+b_{1,1}e_n,\\
\varphi(e_i,e_2) &= (i-n+1)b_{1,1}e_i+b_{2,3}e_{i+1}, \quad 2\leq i\leq n,\\
 \varphi(e_1,e_i) &= -a_{i-1,0}e_1-a_{i-1,0}e_n, \quad 3\leq i\leq n,\\
\varphi(e_i,e_3) &= (n-i-1)a_{2,0}e_i-(a_{2,1}+b_{1,1})e_{i+1}, \quad 2\leq i\leq n-1,\\
\varphi(e_n,e_j) &= -a_{j-1,0}e_n, \quad 3\leq j\leq n,\\
\varphi(e_i,e_j) &= (n-i-1)a_{j-1,0}e_i+(a_{j-2,0}-a_{j-1,1})e_{i+1}, \quad 2\leq i\leq n-1, \quad 4\leq j\leq n,\\
\varphi(e_1,x) &=(a_{n-1,0}-d_{1,0})x +(a_{n-1,1}-a_{n-2,0}-d_{1,1})e_1+\sum\limits_{i=2}^{n-2}(\sum\limits_{t=2}^{i}a_{n+t-i-1,t}-(n-i)a_{1,i+1})e_i\\
 {} &+(\sum\limits_{t=2}^{n-1}a_{t,t}-a_{1,n}+\frac{(n-2)(n-3)}{2}d_{1,0})e_{n-1}+c_{1,n}e_n,\\
\varphi(e_2,x) &
=(n-3)b_{1,1}x + (n-2)b_{2,3}e_1+\sum\limits_{i=2}^{n}c_{2,i}e_i,\displaybreak\\
\varphi(e_3,x) &=(4-n)a_{2,0}x
+(3-n)(a_{2,1}+b_{1,1})e_1+(a_{2,2}+(n-3)d_{1,0})e_2+(c_{2,2}-d_{1,1})e_3\\
{} & +\sum\limits_{i=4}^{n-2}(c_{2,i-1}+(3-i)a_{2,i})e_i +(c_{2,n-2}+(4-n)a_{2,n-1}+a_{2,0})e_{n-1} \\
{} &  +(c_{2,n-1}+(3-n)a_{2,n}-a_{2,1})e_n, \\
\varphi(e_i,x) & =(i-n+1)a_{i-1,0}x+(n-i)(a_{i-2,0}-a_{i-1,1})e_1\\
{} &+\sum\limits_{s=2}^{i-2}(i-s)\sum\limits_{t=2}^{s}a_{i+t-s-1,t}e_s+ (\sum\limits_{t=2}^{i-1}a_{t,t}+\frac{(2n-i-3)(i-2)}{2}d_{1,0})e_{i-1}\\
{} &+(c_{2,2}+(2-i)d_{1,1})e_i+\sum\limits_{s=i+1}^{n-2}(c_{2,s-i+2}+(i-s)\sum\limits_{t=2}^{i-1}a_{t,s-i+t+1})e_s\\
{} &+(c_{2,n-i+1}+a_{i-1,0}+(i-n+1)\sum\limits_{t=2}^{i-1}a_{t,n-i+t})e_{n-1}\\
{} &+(c_{2,n-i+2}+a_{i-2,0}-a_{i-1,1}+(i-n)\sum\limits_{t=2}^{i-1}a_{t,n-i+t+1})e_n, \quad 4\leq i\leq n-2,\\
 \varphi(e_{n-1},x) &=(a_{n-3,0}-a_{n-2,1})e_1+\sum\limits_{s=2}^{n-3}(n-1-s)\sum\limits_{t=2}^{s}a_{n+t-s-2,t}e_s\\
 {}& +\big(\sum\limits_{t=2}^{n-2}a_{t,t}+\frac{(n-2)(n-3)}{2}d_{1,0}\big)e_{n-2}
 +(c_{2,2}+a_{n-2,0}-(n-3)d_{1,1})e_{n-1}\\
 {} & +\big(c_{2,3}+a_{n-3,0}-a_{n-2,1}-\sum\limits_{t=2}^{n-2}a_{t,t+2}\big)e_n,\\
\varphi(e_{n},x) & =a_{n-1,0}x+ \sum\limits_{s=2}^{n-2}(n-s)\sum\limits_{t=2}^{s}a_{n+t-s-1,t}e_s+\big(\sum\limits_{t=2}^{n-1}a_{t,t}+\frac{(n-2)(n-3)}{2}d_{1,0}+a_{n-1,0}\big)e_{n-1}\\
{} &+(c_{2,2}+a_{n-2,0}-a_{n-1,1}-(n-2)d_{1,1})e_n,\\
\varphi(x,e_1) &=d_{1,0}x+d_{1,1}e_1+a_{1,3}e_2+ \sum\limits_{i=3}^{n}d_{1,i}e_i,\\
\varphi(x,e_2)& =-b_{2,3}e_1-b_{1,1}e_{n-1},\\
\varphi(x,e_3) &=(a_{2,1}+b_{1,1})e_1+a_{2,0}e_{n-1},\\
\varphi(x,e_i) &=(a_{i-1,1}-a_{i-2,0})e_1+a_{i-1,0}e_{n-1}, \quad 4\leq i\leq n,\\
\varphi(x,x) & = -a_{n-2,0}x + (a_{n-3,0}-a_{n-2,1})e_1 +
\sum\limits_{i=2}^{n-3}\big((i-n+1)(d_{1,i+1}-a_{1,i+2})-\sum\limits_{t=2}^{i}a_{n-i+t-2,t}\big)e_i\\
{} & +\big(a_{1,n}-d_{1,n-1}-\frac{n^2-5n+10}{2}d_{1,0}-\sum\limits_{t=2}^{n-2}a_{t,t}\big)e_{n-2}+
(a_{n-1,n}+d_{1,1}-c_{1,n})e_{n-1}+\beta_ne_n.
\end{align*}

\end{thm}

\begin{proof}

The proof of this theorem is carrying out by applying similar
arguments as in the proof of Theorem~\ref{thR_3}.

\end{proof}

\begin{rem}\label{remR_4}
 It should be noted that in the case
$n=3$ an arbitrary $2$-cocycle for the algebra $R_4$ is different from the description of Theorem~\ref{thmR_4} and it has the form:
\begin{align*}
\varphi(e_1,e_1) &= a_{1,3}e_3,\\
\varphi(e_2,e_1) &= a_{2,0}x+ \sum\limits_{s=1}^{3}a_{i,s}e_s,\\
\varphi(e_3,e_1) & = a_{2,0}e_1-(a_{2,2}-d_{1,0})e_3,\\
\varphi(e_1,e_2) & = b_{1,1}e_1+b_{1,1}e_3,\\
\varphi(e_2,e_2) & = b_{2,3}e_3,\\
\varphi(e_3,e_2) & = b_{1,1}e_3,\\
\varphi(e_1,e_3) & = -a_{2,0} e_1 - a_{2,0}e_3,\\
\varphi(e_2,e_3) & = -(a_{2,1}+ b_{1,1})e_3,\\
\varphi(e_3,e_3) & = -a_{2,0} e_3,\\
\varphi(e_1,x) & = (a_{2,0} - d_{1,0}) x + (a_{2,1} + b_{1,1}- d_{1,1}) e_1 + (a_{2,2} - a_{1,3}) e_2- c_{1,3}e_3,\\
\varphi(e_2,x) & = b_{2,3}e_1 + c_{2,2} e_2+ c_{2,3}e_3,\\
\varphi(e_3,x) & = a_{2,0} x + (a_{2,2} + a_{2,0}) e_2-(a_{2,1}+ c_{2,2} - d_{1,1})e_3,\\
\varphi(x,e_1) & = d_{1,0} x + d_{1,1} e_1 +  (a_{1,3} -2d_{1,0}) e_2 + d_{1,3}e_3,\\
\varphi(x,e_2) & =  b_{2,3} e_1 - b_{1,1} e_2, \\
\varphi(x,e_3) & =  (a_{2,1}+b_{1,1}) e_1 + a_{2,0} e_2,\\
\varphi(x,x) & =  b_{1,1} x + b_{2,3} e_1 +  (a_{2,3} -c_{1,3}+d_{1,1}) e_2 + \beta_3 e_3.
\end{align*}
\end{rem}

\begin{cor} $\dim ZL^2(R_4, R_4) = (n+1)^2 - 2, $ $\dim HL^2(R_4, R_4) = 1$
 and the adjoint class  $\overline{\rho}$ forms a basis of $HL^2(R_4,
R_4)$, where
\[\rho: \left\{\begin{array}{l}
\rho(e_1,x) = e_1, \\[1mm]
\rho(e_i,x) = (i-2)e_i,\ 3\leq i\leq n, \\[1mm]
\rho(x,e_1) = -e_1,\\[1mm]
\rho(x,x) = -e_{n-1}.
\end{array} \right.\]
\end{cor}

\begin{thm}
 An arbitrary $\varphi\in ZL^2\big(R_2(\alpha), R_2(\alpha)\big)$ has the following form:
\begin{align*}
\varphi(e_1,e_1) &= \sum\limits_{i=2}^{n}a_{1,i}e_i,\\
\varphi(e_i,e_1) &= \sum\limits_{s=1}^{n}a_{i,s}e_s+a_{i,0}x, \ 2\leq i\leq n-1, \\
\varphi(e_n,e_1) &=a_{n-1,0}e_1-\sum\limits_{i=3}^{n-1}\big(\sum\limits_{t=2}^{i-1}a_{n+t-i,t}\big)e_i+\Big(\frac{(n-1)(2\alpha+n)}{2}d_{1,0}-\sum\limits_{t=2}^{n-1}a_{t,t}\Big)e_n,\\
\varphi(e_1,e_2) &=b_{1,1}e_1+b_{1,n}e_n,\\
\varphi(e_i,e_2) &= \big((i-2)b_{1,1}+b_{2,2}\big)e_i+b_{2,3}e_{i+1}, \quad 2\leq i\leq n, \\
\varphi(e_1,e_i) &=-a_{i-1,0}e_1, \ 3\leq i\leq n,\\
\varphi(e_i,e_3) &= -(\alpha+i-1)a_{2,0}e_i-(a_{2,1}+b_{1,1})e_{i+1}, \quad 2\leq i\leq n,  \\
\varphi(e_i,e_j) &= -(\alpha+i-1)a_{j-1,0}e_i+(a_{j-2,0}-a_{j-1,1})e_{i+1}, \quad 2\leq i\leq n, \quad 4\leq j\leq n, \displaybreak \\
\varphi(e_1,x) &=-d_{1,1}e_1+\sum\limits_{i=2}^{n-1}(\alpha+i-2)a_{1,i+1}e_i+c_{1,n}e_n-d_{1,0}x,\\
\varphi(e_2,x) &=-\alpha b_{2,3}e_1+\sum\limits_{i=2}^{n}c_{2,i}e_i-(\alpha+1)b_{1,1}x,\\
\varphi(e_3,x) &=(\alpha+1)(a_{2,1}+b_{1,1})e_1+(a_{2,2}-(\alpha+1)d_{1,0})e_2+(c_{2,2}-d_{1,1})e_3\\
{} &+\sum\limits_{i=4}^{n}(c_{2,i-1}+(3-i)a_{2,i})e_i+(\alpha+2)a_{2,0}x,\\
\varphi(e_i,x) & = (\alpha+i-2)(a_{i-1,1}-a_{i-2,0})e_1+\sum\limits_{s=2}^{i-2}(i-s)\sum\limits_{t=2}^{s}a_{i+t-s-1,t}e_s \\
{}& + \Big(\sum\limits_{t=2}^{i-1}a_{t,t}-\frac{(2\alpha+i-1)(i-2)}{2}d_{1,0}\Big)e_{i-1}+(c_{2,2}-(i-2)d_{1,1})e_i\\
{}&+\sum\limits_{s=i+1}^{n}(c_{2,s-i+2}+(i-s)\sum\limits_{t=2}^{i-1}a_{t,s-i+t+1})e_s+(\alpha+i-1)a_{i-1,0}x, \quad 4\leq i\leq n, \\
\varphi(x,e_1) &=\sum\limits_{i=1}^{n}d_{1,i}e_i+d_{1,0}x,\\
\varphi(x,e_2) &=-b_{2,3}e_1-b_{1,n}e_{n-1},\\
\varphi(x,e_3) &=(a_{2,1}+b_{1,1})e_1, \\
\varphi(x,e_i) &=(a_{i-1,1}-a_{i-2,0})e_1, \quad 4\leq i\leq n, \\
\varphi(x,x) &= \sum\limits_{i=2}^{n-2}(\alpha+i-1)(d_{1,i+1}-a_{1,i+2})e_i+\big((\alpha+n-2)d_{1,n}-c_{1,n}\big)e_{n-1}+\beta_ne_n
\end{align*}
with restrictions
\begin{equation}
\label{eq3.1}\left\{\begin{array}{ccc}(\alpha-1)a_{1,2}=0,&
 (\alpha+1)b_{1,n}=0, &
(\alpha+1)(b_{2,2}-(\alpha+1)b_{1,1})=0,\\[1mm]
(n-3)b_{1,n}=0, & \alpha (d_{1,2}-
a_{1,3})=0.\end{array}\right.
\end{equation}
\end{thm}

\begin{proof}
The proof of this proposition is carrying out by applying similar
arguments as in the proof of Theorem~\ref{thR_3}.
\end{proof}

From the equalities \eqref{eq3.1} it implies that if $n\neq 3$,
then $b_{1,n} =0$. Thus, we distinguish the cases $n=3$ and $n>3$.
Moreover, the general form of infinitesimal deformations also depends
on the value of $\alpha$. Therefore, we have
\begin{cor}
\begin{align*}
\dim ZL^2\big(R_2(\alpha),R_2(\alpha)\big) & = \left\{\begin{array}{ll}
                (n+1)^2 - 1,& \alpha=0;\pm 1,\\
                (n+1)^2 - 2,& \alpha\neq 0; \pm 1,
\end{array}\right. \quad \ for \ n > 3;\\
\dim ZL^2\big(R_2(\alpha),R_2(\alpha)\big) &  = \left\{\begin{array}{ll}
                15,& \alpha=0; 1,\\
                16,& \alpha= -1,\\
                14,& \alpha\neq 0;\pm 1,
\end{array}\right.\quad \ for \ n = 3.
\end{align*}
\end{cor}

\begin{cor}
\begin{align*}
\dim HL^2\big(R_2(\alpha),R_2(\alpha)\big) & = \left\{\begin{array}{ll}
                2,& \alpha=0;\pm 1; 1-n; 2-n,\\
                1,& \alpha\neq 0;\pm 1; 1-n; 2-n,
\end{array}\right. \quad \ for \ n > 3;\\
\dim HL^2\big(R_2(\alpha),R_2(\alpha)\big) & = \left\{\begin{array}{ll}
                4,& \alpha= - 1,\\
                2,& \alpha= 0; 1; -2,\\
                1,& \alpha\neq 0; \pm 1; -2,
\end{array}\right.\quad \ for \ n = 3.
\end{align*}

\end{cor}

In the following proposition similarly to the
proof of Proposition~\ref{prop3.10} we find a basis of
$HL^2\big(R_2(\alpha),R_2(\alpha)\big)$.

\begin{prop}
The basis of $HL^2\big(R_2(\alpha),R_2(\alpha)\big)$ consists of the following
adjoint classes
\[ \left\{\begin{array}{ll} \overline{\rho}, \overline{\psi_1} & \alpha = 0;1; 1-n; 2-n, \\[1mm]
\overline{\psi_1}, \overline{\psi_2}, & \alpha = -1, \\[1mm]
\overline{\rho}, & \alpha \neq 0; \pm 1; 1-n; 2-n, \\[1mm]
 \end{array}\right. \quad \ for \ n>3; \]
\[ \left\{\begin{array}{ll} \overline{\rho}, \overline{\psi_1} & \alpha = 0;-1; -2, \\[1mm]
\overline{\psi_1}, \overline{\psi_2},  \overline{\psi_3}, \overline{\psi_4}, & \alpha = -1, \\[1mm]
\overline{\rho}, & \alpha \neq 0; \pm 1; -2, \\[1mm]
 \end{array}\right. \quad \ for \ n=3, \]
where
\[\rho: \left\{\begin{array}{ll}
\rho(e_1,x) = e_1, \\[1mm]
\rho(e_i,x) = (i-2)e_i, & 3\leq i\leq n, \\[1mm]
\rho(x,e_1) = -e_1,
\end{array} \right.\]
\[\begin{array}{ll}\psi_1:\psi_1(e_i,x) = e_{i}, \quad 2\leq i\leq n, &  \psi_2:
\psi_2(e_i,e_2) = e_i,  \quad 2\leq i\leq n, \\ [1mm] \psi_3:
\left\{\begin{array}{l}
\psi_3(e_1,x) = e_n,  \\[1mm]
\psi_3(x,x) = - e_{n-1},
\end{array} \right. &  \psi_4: \left\{\begin{array}{l}
\psi_4(e_1,e_2) = e_n,  \\[1mm]
\psi_4(x,e_2) = -e_{n-1}.
\end{array} \right. \end{array}\]
\end{prop}

\begin{thm}
 An arbitrary $2$-cocycle of $ZL^2\big(R_5(\alpha_{4},\alpha_{5}, \dots,\alpha _{n}),R_5(\alpha_{4},\alpha_{5}, \dots,\alpha _{n})\big)$ has the following form:
 \begin{align*}
\varphi(e_1,e_1) &= a_{2,0} x+a_{2,1} e_{1} +(a_{2,2} -\alpha_{n} a_{n,2} )e_{2} +\sum\limits_{i=3}^{n}a_{1,i} e_{i},\\
\varphi (e_{i} ,e_{1} )& =a_{i,0} x+\sum\limits_{s=1}^{n}a_{i,s} e_{s}, \quad  2 \leq i \leq n-1,\\
\varphi (e_{n} ,e_{1} )& =\sum\limits_{i=2}^{n}a_{n,i} e_{i},\\
\varphi (e_{1} ,e_{2} )&= (-c_{0} +\sum\limits_{j=4}^{n}\alpha_{j} a_{j-1,0})e_{2} +(-c_{1}
+\sum\limits_{j=4}^{n}\alpha _{j} a_{j-1,1})e_{3} +\sum\limits_{i=4}^{n-1}\alpha_{i} (-c_{0} +\sum\limits_{j=4}^{n}\alpha _{j} a_{j-1,0} ) e_{i}, \\
\varphi (e_{i} ,e_{2} )& =(-c_{0} +\sum\limits_{j=4}^{n}\alpha_{j}
a_{j-1,0})e_{i} +(-c_{1} +\sum\limits_{j=4}^{n}\alpha_{j}
a_{j-1,1} )e_{i+1} +\sum\limits_{k=i+2}^{n}\alpha_{i} (-c_{0}
+ \sum\limits_{j=4}^{n}\alpha_{j} a_{j-1,0} ) e_{k}, \quad 2 \le i\le n,\\
\varphi (e_{1} ,e_{j} )&=-a_{j-1,0} e_{2} - a_{j-1,1} e_{3}
-a_{j-1,0} \sum\limits_{k=4}^{n-1}\alpha _{k} e_{k},  \quad  3\le j\le n,\\
\varphi (e_{i} ,e_{j} )&= -a_{j-1,0} e_{i} -a_{j-1,1} e_{i+1}
-a_{j-1,0} \sum\limits_{k=i+2}^{n}\alpha _{k-i+2} e_{k},  \quad 2\le i\le n, \quad 3\le j\le n, \\
\varphi (e_{1} ,x) & = (c_{0} -\alpha _{n} a_{n-1,0} )x+(c_{1} -d_{2} -\alpha _{n} a_{n-1,1} )e_{1} \\
{} & +(c_{2} +a_{1,3} -a_{2,3} +d_{2} +\alpha _{n} a_{n-1,1} +\alpha _{n} a_{n,3} )e_{2}
+\big(c_{3} +a_{1,4} -a_{2,4} +\alpha _{n} (a_{n,3} -\alpha _{4} a_{n,2})\big)e_{3} \\
{} & + \sum\limits_{i=4}^{n-1}(c_{i} +a_{1,i+1} -a_{2,i+1} +\alpha_{n} (a_{n,i} -\alpha _{i+1} a_{n,2} )
+\sum\limits_{j=4}^{i}\alpha_{j} (a_{1,i-j+3} -a_{2,i-j+3} ) ) e_{i} +c_{n+1} e_{n},\displaybreak \\
\varphi (e_{2},x) & =c_{0} x+\sum\limits_{i=1}^{n}c_{i} e_{i},\\
\varphi (x,e_{1} ) & =-d_{2} e_{1} +\sum\limits_{i=2}^{n}d_{i} e_{i}, \\
\varphi (e_{3},x)& =(a_{2,0} +\sum\limits_{j=4}^{n}\alpha _{j}a_{j,0} )x+(a_{2,1}
+\sum\limits_{j=4}^{n-1}\alpha _{j} a_{j,1})e_{1} +(-a_{2,1} +\sum\limits_{j=4}^{n}\alpha _{j} a_{j,2})e_{2}\\
{} & +(c_{1} +c_{2} +\sum\limits_{j=4}^{n}\alpha _{j} a_{j,3}  )e_{3}
+\sum\limits_{k=4}^{n}(c_{k-1} -a_{2,1} \alpha _{k}
-\sum\limits_{j=2}^{k-2}\alpha _{k-j+2} a_{2,j}
+\sum\limits_{j=4}^{n}\alpha _{j} a_{j,k}  )e_{k}  +a_{2,1} \alpha
_{n} e_{n}, \\
\varphi (e_{i+1} ,x)& =(a_{i,0} +\sum\limits_{j=i+2}^{n}\alpha_{j-i+2} a_{j,0}  )x+(a_{i,1}
+\sum\limits_{j=i+2}^{n-1}\alpha_{j-i+2} a_{j,1} )e_{1} +(-a_{i,1}
+\sum\limits_{j=i+2}^{n}\alpha_{j-i+2} a_{j,2} )e_{2}\\
{} & +\sum\limits_{k=3}^{i-1}(\sum\limits_{j=i+1}^{n-1}\alpha _{j-i+k}a_{j,1}
+\sum\limits_{s=2}^{k-1}\sum\limits_{j=i+1}^{n}\alpha_{j+k-i-s+2} a_{j,s} +\sum\limits_{j=i+2}^{n}\alpha_{j-i+2}a_{j,k})e_{k}\\
{} & +(\sum\limits_{j=i+1}^{n-1}\alpha _{j} a_{j,1}
+\sum\limits_{s=2}^{i-1}\sum _{j=i+1}^{n}\alpha _{j-s+2} a_{j,s}
+\sum\limits_{j=i+2}^{n}\alpha _{j-i+2} a_{j,i})e_{i} \\
{} & +\big(c_{1} +c_{2} -(i-2)d_{2} -\sum\limits_{j=3}^{i}\alpha _{j+1}(a_{j,1} +a_{j,2} )
+\sum\limits_{s=3}^{i}\sum\limits_{j=i+1}^{n}\alpha_{j-s+3}
a_{j,s} +\sum\limits_{j=i+2}^{n}\alpha_{j-i+2} a_{j,i+1} \big) e_{i+1} \\
{} & +\sum\limits_{k=i+2}^{n}\big(c_{k+1-i} +
\sum\limits_{j=2}^{i}\alpha _{k-i+j} a_{j,1} -\sum\limits_{s=2}^{i}\sum\limits_{j=2}^{k-s}\alpha _{k+4-j-s}a_{i-s+2,j}
+\sum\limits_{s=2}^{i}\sum\limits_{j=s+2}^{n}\alpha_{j-s+2} a_{j,k-i+s} \big) e_{k} \\
{} & +\alpha _{n} a_{i,1} e_{n},\\
 \varphi (x,x)&=d_{3} e_{2} +d_{4} e_{3} +\sum\limits_{i=4}^{n-2}(d_{i+1} +
 \sum\limits_{j=4}^{i}\alpha _{j} d_{i+3-j}  )e_{i}
+\big(d_{n} +\sum\limits_{j=4}^{n}\alpha _{j} d_{n+2-j}\big)e_{n-1} +\beta e_{n}.
\end{align*}
\end{thm}

\begin{proof}
The proof of this proposition is carrying out by applying similar
arguments as in the proof of Theorem~\ref{thR_3}.
\end{proof}

\begin{cor}
\begin{align*}
\dim ZL^2\big(R_5(\alpha_i), R_5(\alpha_i)\big) & = n^2 +3n -3,\\
 \dim HL^2\big(R_5(\alpha_i), R_5(\alpha_i)\big) & =
 \left\{\begin{array}{ll}2n-4, & \alpha_i = 0 \ \text{for  all} \ i,   \\
2n-5, & \alpha_i \neq 0 \  \text{for  some}   \ i.
\end{array}\right.
\end{align*}
\end{cor}

Let us introduce the notations
\begin{align*}
\rho: &  \left\{\begin{array}{ll}
\rho(e_1,x) = e_1 - e_2, \\[1mm]
\rho(e_i,x) = (i-3)e_i, & 4\leq i\leq n, \\[1mm]
\rho(x, e_1) = e_1 - e_2,
\end{array} \right.\\
\psi_k  \, (4 \leq k \leq n-1 ): & \left\{\begin{array}{ll} \psi_k(e_1,x) = e_k,\\[1mm]
\psi_k(e_i,x) =  e_{k+i-2},& 2\leq i\leq n-k+2,
\end{array} \right.\\
\psi_n: &  \left\{\begin{array}{ll} \psi_n(e_2,x) = e_n,
\end{array} \right.\\
\varphi_{n,2}: & \left\{\begin{array}{ll}
\varphi_{n,2}(e_n,e_1) = e_2, \\[1mm]
\varphi_{n,2}(e_1,x) = -\alpha_n \sum\limits_{j=3}^{n-1}\alpha_{j+1}e_j,\\[1mm]
\varphi_{n,2}(e_i,x) = \sum\limits_{j=2}^{i-1}\alpha_{n+j+1-i}e_j, & 3\leq i\leq n-1, \\[1mm]
\varphi_{n,2}(e_n,x) = \sum\limits_{j=3}^{n-1}\alpha_{j+1}e_j,
\end{array} \right. \displaybreak \\
\varphi_{n,k} \, (3 \leq k \leq n-1): & \left\{\begin{array}{ll}
\varphi_{n,k}(e_n,e_1) = e_k, \\[1mm]
\varphi_{n,k}(e_1,x) = -\alpha_n e_k,\\[1mm]
\varphi_{n,k}(e_i,x) = \sum\limits_{j=k}^{i+k-3}\alpha_{n+j+3-i-k}e_j, & 3\leq i\leq n+2-k, \\[1mm]
\varphi_{n,k}(e_i,x) = \sum\limits_{j=k}^{n}\alpha_{n+j+3-i-k}e_j, & n+3-k\leq i\leq n-1, \\[1mm]
\varphi_{n,k}(e_n,x) = \sum\limits_{j=k+1}^{n}\alpha_{j+3-k}e_j.
\end{array} \right.
\end{align*}

\begin{prop}
The adjoint classes  $\overline{\rho}$, $\overline{\psi_k} \ (4\leq k \leq n) $ and
$\overline{\varphi_{n,k}} \ (2 \leq k \leq n-1)$ form a basis of $HL^2\big(R_5(0), R_5(0)\big)$.
 The basis of
$HL^2\big(R_5(\alpha_4, \dots, \alpha_n), R_5(\alpha_4, \dots,\alpha_n)\big)$ with $(\alpha_4, \dots, \alpha_n) \neq (0, \dots, 0)$,
is also the same except one cocycle $\overline{\psi_k}$ with
$\alpha_k \neq 0$.
\end{prop}

\begin{proof} Since there are parameters $(a_{i,j}, c_{k},  \beta, d_{s})$ in
the general form of $2$-cocycles for the algebra $R_5(\alpha_4, \dots, \alpha_n)$, we consider the natural basis of
the space $ZL^2$ whose basis elements are obtained by the
instrumentality of these parameters.

Similarly as in the proof of Proposition~\ref{prop3.10}, we denote by
$\varphi(a_{i,j})$ the cocycle which satisfies  $a_{i,j} =1$ and all other
parameters  are equal to zero. We define such type of notation for other
parameters by notations
\[\varphi_{i,j}= \varphi(a_{i,j}), \quad \psi_{k}= \varphi(c_{k}),
 \quad \rho_{s}= \varphi(d_{1,s}), \quad \eta = \varphi(\beta),\]
where $1\leq i \leq n $, $0\leq j \leq n$, $ 0\leq k \leq n+1 $,
$2\leq s \leq n$ and
\[(i,j) \notin \{(1,0),  (1,1),  (1,2)\}.\]

To define the basis of $BL^2\big(R_5(\alpha_i),R_5(\alpha_i)\big)$
we consider the endomorphisms $f_{j,k} : R_5(\alpha_i)\rightarrow
R_5(\alpha_i)$ defined as follow
\begin{align*}
f_{i,j}(e_i) & =e_j,  \qquad 1 \leq i, j \leq n,\\
f_{i,n+1}(e_i) & =x, \qquad  \ 1 \leq i \leq n,\\
f_{n+1,j}(x) & = e_j, \qquad  1 \leq j \leq n,\\
f_{n+1,n+1}(x) &  = x,
\end{align*}
where in the expansion of endomorphisms the omitted values are assumed to be zero.

Consider  \[g_{i,j}(x,y) = [f_{i,j}(x),y] + [x,f_{i,j}(y)] -
f_{i,j}([x,y]).\]

Note that $g_{i,j} \in BL^2\big(R_5(\alpha_i), R_5(\alpha_i)\big)$ and by direct computation we
express $g_{i,j}$ via the elements $\varphi_{i,j}, \psi_{k},
\rho_{s}$ and $\eta$.

In the case of $\alpha_4=\alpha_5 = \dots = \alpha_n =0$ we obtain
that any linear combination of elements $\rho_{2}$, $\psi_{k}, 4
\leq k \leq n$ and $\varphi_{n,j}, 2 \leq j \leq n-1$ does not
belong to $BL^2\big(R_5(\alpha_i), R_5(\alpha_i)\big)$.
Therefore, the adjoint classes of these
elements form a basis of $HL^2\big(R_5(\alpha_i),R_5(\alpha_i)\big)$.

However, in the case of $(\alpha_4,\alpha_5, \dots, \alpha_n) \neq (0,
0, \dots 0)$ we obtain that \[2\alpha_4\psi_{4} +3\alpha_5\psi_{5}
+\dots+ (n-2)\alpha_n\psi_{n} \in BL^2\big(R_5(\alpha_i), R_5(\alpha_i)\big).\]

Hence, in this case we get that the basis of $HL^2\big(R_5(\alpha_i), R_5(\alpha_i)\big)$ also consists
from $\overline{\rho_2}$, $\overline{\psi_k} \ (4 \leq k \leq n) $
and $\overline{\varphi_{n,k}} \ (2 \leq k \leq n-1) $, except one
cocycle $\overline{\psi_k}$ with $\alpha_k \neq 0$.
\end{proof}

\section*{Acknowledgments}

This work was partially supported by Ministerio de Econom\'ia y Competitividad (Spain),
grant MTM2013-43687-P (European FEDER support included); by Xunta de Galicia, grant GRC2013-045 (European FEDER support included) and by the Grant No.0251/GF3 of Education and Science Ministry
of Republic of Kazakhstan.

\end{document}